\documentclass[12pt]{article}

\usepackage{amsmath, amsthm, amscd, amsfonts, amssymb, graphicx}
\usepackage[T1]{fontenc}

\newtheorem{theorem}{Theorem}[section]

\newtheorem{lemma}[theorem]{Lemma}
\newtheorem{proposition}[theorem]{Proposition}
\newtheorem{observation}[theorem]{Observation}

\newcommand{\gp}{\mathop{\mathrm{gp}}}
\newcommand{\mut}{\mathop{\mu_\mathrm{t}}}
\newcommand{\muo}{\mathop{\mu_\mathrm{o}}}
\newcommand{\diam}{\mathop{\mathrm{diam}}}

\title{Mutual-visibility and general position in double graphs and in Mycielskians}

\author{
Dhanya Roy $^{a}$\footnote{\tt dhanyaroyku@gmail.com, dhanyaroyku@cusat.ac.in} 
\and Sandi Klav\v{z}ar $^{b,c,d}$\footnote{\tt sandi.klavzar@fmf.uni-lj.si} 
\and Aparna Lakshmanan S $^{a}$\footnote{\tt aparnaren@gmail.com, aparnals@cusat.ac.in}\\\\
$^{a}$ \small Department of Mathematics, 	Cochin University of Science and Technology, 
\\ \small Cochin - 22, Kerala, India\\
$^{b}$\small Faculty of Mathematics and Physics, University of Ljubljana, Slovenia\\
$^{c}$ \small Institute of Mathematics, Physics and Mechanics, Ljubljana, Slovenia \\
$^{d}$ \small Faculty of Natural Sciences and Mathematics, University of Maribor, Slovenia\\
}

\date{\today}

\begin{document}
\maketitle

\begin{abstract}
The general position problem in graphs is to find the maximum number of vertices that can be selected such that no three vertices lie on a common shortest path. The mutual-visibility problem in graphs is to find the maximum number of vertices that can be selected such that every pair of vertices in the collection has a shortest path between them with no vertex from the collection as an internal vertex. In this paper, the general position problem and the mutual-visibility problem is investigated in double graphs and in Mycielskian graphs. Sharp general bounds are proved, in particular involving the total mutual-visibility number and the outer mutual-visibility number of base graphs. Several exact values are also determined, in particular the mutual-visibility number of the double graphs and of the Mycielskian of cycles.
\end{abstract}


\medskip\noindent
{\bf AMS Subj.\ Class.\ (2020)}: 05C12, 05C69, 05C76

\section{Introduction}
	
The graph general position problem reflects the Dudeney's no-three-in-line problem~\cite{dudeney-1917} as well as the general position subset selection problem from discrete geometry~\cite{froese-2017}. The problem was in a different context investigated on hypercubes~\cite{korner-1995}, while it was introduced in its generality as follows~\cite{manuel-2018}. A set $S$ of vertices in a graph is a {\em general position set} if no three vertices from $S$ lie on a common shortest path. A largest general position set of a graph $G$ is called a {\em gp-set} of $G$ and its size is the {\em general position number} $\gp(G)$ of $G$. The same concept was in use two years earlier in~\cite{chandran-2016} under the name geodetic irredundant sets, where it was defined in a different way.
	
In discrete geometry, a shortest path between two points is unique while in graphs there can be more than one shortest path between two vertices. This fact, as well as the computational concept of visibility between robots bring the mutual-visibility problem in graphs into picture. This problem was introduced by Di Stefano~\cite{distefano-2022} as follows. Given a set $S$ of vertices in a graph $G$, two vertices $u$ and $v$ are {\em mutually-visible} or, more precisely, {\em $S$-visible}, if there exists a shortest $u,v$-path in $G$ which contains no further vertices from $S$. The set $S$ is a {\em mutual-visibility set} if its vertices are pairwise mutually-visible. A largest mutual-visibility set is called a {\em $\mu$-set} and its size is called the {\em mutual-visibility number} $\mu(G)$ of $G$.

In~\cite{cicerone-2023a}, a variety of mutual-visibility sets was introduced, we will use the following two variants. A set $S$ is an {\em outer mutual-visibility set} in $G$ if $S$ is a mutual-visibility set and every pair of vertices $u \in S$, $v \in V(G)\setminus S$ are also $S$-visible. A largest outer mutual-visibility set is called  a {\em $\muo$-set}. The size of a largest outer mutual-visibility set is called the {\em outer mutual-visibility number} of $G$, denoted as $\muo(G)$. The set $S$ is a {\em total mutual-visibility set} in $G$ if every pair of vertices in $G$ are $S$-visible. A largest total mutual-visibility set is called  a {\em $\mut$-set}. The size of a largest total mutual-visibility set is called the {\em total mutual-visibility number} of $G$, denoted as $\mut(G)$.	
	
The general position problem and the mutual-visibility problem are well studied for different graph classes like diameter two graphs~\cite{anand-2019, cicerone-2024+}, cographs~\cite{anand-2019, distefano-2022}, Kneser graphs~\cite{ghorbani-2021}, and line graphs of complete graphs~\cite{cicerone-2024+, ghorbani-2021}. Both problems were also investigated a lot on graph operations like the join of graphs~\cite{distefano-2022, ghorbani-2021}, corona products~\cite{cicerone-2023b, ghorbani-2021, klavzar-2019}, Cartesian products~\cite{cicerone-2023b, klavzar-2021, klavzar-2021b, tian-2021a, tian-2021b}, and strong products~\cite{cicerone-2023+, klavzar-2019}. In this paper we extend this line of research by investigating the problems on double graphs and on Mycielskian graphs which are respectively defined as follows.

Let $G$ be a graph. The {\em double graph} $D(G)$ of $G$ is constructed from the disjoint union of $G$ and an isomorphic copy $G'$ of $G$, where $V(G') = \{u':\ u\in V(G)\}$, by joining $u\in V(G)$ to all neighbors of $u'\in V(G')$, and joining $u'\in V(G')$ to all neighbors of $u\in V(G)$. That is, $V(D(G)) = V(G) \cup V(G')$ and for each pair $u\in V(G)$ and $u'\in V(G')$ we have $N_{D(G)}(u) = N_{D(G)}(u')$, where $N_H(u)$ denotes the open neighborhood of the vertex $u$ in a graph $H$. The {\em Mycielskian graph} $M(G)$ of $G$ has the vertex set $V(M(G)) = V(G) \cup V(G') \cup \{v^*\}$, where $V(G') = \{u':\ u\in V(G)\}$ and the edge set  $E(M(G)) = E(G) \cup \{uv':\ uv \in E(G)\} \cup \{v'v^*:\ v' \in V(G')\}$.  These two graph operators were respectively introduced in~\cite{munarini-2008, mycielski-1955}. The Mycielskian has been studied in a couple of hundred papers and the trend is still continuing~\cite{bidine-2023, boutin-2024, james-2023, kalarkop-2024}. The double graphs have also received quite some attention, cf.~\cite{kozorenko-2023, lou-2022}. 

The paper is organized as follows. In the next section additional definitions required for this paper are listed, known results recalled, and some new observations stated. In Section~\ref{sec:mutual-visibility-double} we prove that if $G$ is not complete, then $\mu(D(G))\ge n(G) + \mut(G)$. The bound is sharp as in particular follows from the proved formula $\mu(D(C_n)) = n$, $n\ge 7$. On the other hand, we construct graphs $G$ for which the difference $\mu(D(G)) - (n(G) - \mut(G))$ is arbitrary large. In Section~\ref{sec:gp-double} we prove that  $\gp(G) \leq \gp(D(G)) \leq 2\gp(G)$ and that the bounds are sharp. In the subsequent section we consider mutual-visibility in Mycielskian graph. In the main theorems we prove that  $\mu(M(P_n)) = n +  \lfloor \frac{n+1}{4}\rfloor$ for $n\ge 5$, and that $\mu(M(C_n)) =  n + \lfloor \frac{n}{4} \rfloor$ for $n\ge 8$. We also give bounds for $\mu(M(G))$, where $\diam(G) \leq 3$, in terms of $\muo(G)$ and $\mu(G)$.
	
\section{Preliminaries}
	
If $G$ is a connected graph, $S\subseteq V(G)$, and ${\cal P} = \{S_1, \ldots, S_t\}$ a partition of $S$, then ${\cal P}$ is \emph{distance-constant} if for any $i,j\in [t]$, $i\ne j$, the distance $d_G(x,y)$, where $x\in S_i$ and $y\in S_j$, is independent of the selection of $x$ and $y$. This distance is then also the distance $d_G(S_i,S_j)$ between $S_i$ and $S_j$. A distance-constant partition ${\cal P}$ is {\em in-transitive} if $d_G(S_i, S_k) \ne d_G(S_i, S_j) + d_G(S_j,S_k)$ holds for $i,j,k\in [p]$. The following characterization of general position sets will be used either implicitly or explicitly in the rest of the paper. By $G[S]$ we denote the subgraph of $G$ induced by $S\subseteq V(G)$. 

\begin{theorem} {\rm \cite[Theorem 3.1]{anand-2019}}
\label{thm:gpsets}
Let $G$ be a connected graph. Then $S\subseteq V(G)$ is a general position set if and only if the components of $G[S]$ are complete subgraphs, the vertices of which form an in-transitive, distance-constant partition of $S$. 
\end{theorem}	
	
Let $G$ be a graph. Vertices $x$ and $y$ of $G$ are {\em false twins} if $N_G(x) = N_G(y)$. (Note that false twins are not adjacent.) Further, $x$ and $y$ are {\em true twins} if $N_G[x] = N_G[y]$, where $N_G[x]$ denotes the closed neighborhood of the vertex $x$ in $G$. In~\cite{klavzar-2019}, relations between true twins, the general position number, and strong resolving graphs were investigated. The following easy but useful general properties of twins hold. 

\begin{lemma}
\label{lem:false-twins}
Let $G$ be a graph and $u,v\in V(G)$. 

(i) If $u$, $v$ are false twins, and $S$ is a general position (resp.\ mutual-visibility) set of $G$ such that $S\cap \{u,v\} = \{u\}$, then $(S \setminus \{u\}) \cup \{v\}$ is also a general position (resp.\ mutual-visibility) set of $G$. 

(ii) If $u$, $v$ are true twins, and $S$ is a general position set of $G$ such that $u \in S$, then $S \cup \{v\}$ is also a general position set of $G$. 
\end{lemma}
	
\begin{proof}
(i) Since $d_G(u,x) = d_G(v,x)$ for each $x\in V(G)\setminus \{u,v\}$, Theorem~\ref{thm:gpsets} yields that $(S \setminus \{u\}) \cup \{v\}$ is a general position set of $G$. Moreover, two vertices are $S$-visible if and only they are $(S \setminus \{u\}) \cup \{v\}$-visible, hence $(S \setminus \{u\}) \cup \{v\}$ is a mutual-visibility set provided that $S$ is a mutual-visibility set. 

(ii) Using the fact that $d_G(u,x) = d_G(v,x)$ for each $x\in V(G)\setminus \{u,v\}$, Theorem~\ref{thm:gpsets} again can be used to deduce that $S \cup \{v\}$ is a general position set. Indeed, if $v\in S$, there is nothing to prove. Let now $v\notin S$ and let $Q$ be the complete subgraph cotaining $u$ from the partition of $S$ corresponding to Theorem~\ref{thm:gpsets}. Then $Q \cup \{v\}$ is also complete, therefore Theorem~\ref{thm:gpsets} implies that $S \cup \{v\}$ is a general position. 
\end{proof}

Note that Lemma~\ref{lem:false-twins}(ii) does not hold if general position sets are replaced by mutual-visibility sets. For example, consider the complete graph $K_4$ minus an edge with the vertex set $\{a,b,c,d\}$, where $a$ and $b$ are the non-adjacent pair. Then $S = \{a,b,c\}$ is a mutual-visibility set, $c$ and $d$ are true twins, but we cannot add $d$ to $S$ without affecting the mutual-visibility.

\section{Mutual-visibility in double graphs}
\label{sec:mutual-visibility-double}

In this section we consider mutual-visibility in double graphs. For this task recall that if $G$ is a graph, then $V(D(G)) = V(G) \cup V(G')$ and that for each pair $u\in V(G)$ and $u'\in V(G')$ we have $N_{D(G)}(u) = N_{D(G)}(u')$.	
	
If $G$ is a graph, then, clearly, $\mu(G) = n(G)$ if and only if $G$ is a complete graph, the same conclusion holds for the total mutual-visibility~\cite{kuziak-2023+}. (Here and later, $n(G)$ denotes the order of $G$.) For each vertex $u \in V(D(G))$, clearly, $N_{D(G)}[u]$ is a mutual-visibility set of $D(G)$. Thus $\mu (D(G)) \geq 2\Delta(G) + 1$. Hence for the double graphs of graphs with a universal vertex, we have the following observation:
 
\begin{observation}
If $G$ is a graph with $n(G)\ge 2$ and with a universal vertex, then $\mu(D(G)) = 2n(G)-1$. 
\end{observation}
	
\begin{theorem}
\label{thm:mu-of-double-graphs}
If $G$ is not a complete graph, then $\mu (D(G)) \geq n(G) + \mut(G)$. Moreover, the bound is sharp. 
\end{theorem}

\begin{proof}
	Consider an arbitrary $\mut$-set $S$ of $G$. We claim that $X = V(G') \cup S$ is a mutual-visibility set of $D(G)$. Since $S$ is a total mutual-visibility set of $G$, for any two vertices in $G$, there exists a shortest path whose internal vertices are in $V(G)\setminus S$. Hence any two vertices $u, v\in V(G') \cup S$, where $v\ne u'$, there exists a shortest $u,v$-path whose internal vertices are in $V(G)\setminus S$. Moreover, since $N_G[u] \subseteq S$ if and only if $G$ is complete, for each $u \in S$ there exists a vertex $v \in N_G(u)$ such that $v \in V(G)\setminus S$. Hence the path $u-v-u'$ demonstrates that $u$ and $u'$ are also $X$-visible whenever $u,u'\in X$. Thus $V(G') \cup S$ is indeed a mutual-visibility set of $D(G)$ which proves that $\mu (D(G)) \geq n + \mut (G)$.
	
To demonstrate the sharpness, consider the path graph $P_n$, $n\ge 3$, with the vertices $u_1, \ldots, u_n$. Let $S$ be an arbitrary mutual-visibility set of $D(P_n)$. If we would have indices $i< j < k$, such that $\{u_i, u_i', u_j, u_j', u_k, u_k'\}\in S$, then $u_i$ and $u_k$ would not be $S$-visible. Therefore, for at most two indices $i\in [n]$ we have $|S\cap \{u_i, u_i'\}| = 2$ which in turn implies that  $\mu (D(P_n)) \le n + 2$. Since $\mut(P_n) = 2$, the above proved bound yields $\mu (D(P_n)) \ge n + 2$, hence the bound is sharp. 
\end{proof}	

In the seminal paper on the mutual-visibility~\cite{distefano-2022} it was proved that the mutual-visibility problem is NP-complete, while in~\cite{cicerone-2023a} the same conclusion was obtained for each of the problems from the variety of mutual-visibility problems including the total mutual-visibility problem. Theorem~\ref{thm:mu-of-double-graphs} could indicate that the mutual-visibility problem is difficult also when restricted to double graphs. 

The next result yields another family for which the bound of Theorem~\ref{thm:gp-of-double} is sharp. 

\begin{theorem}
\label{thm:mu-of-double-graphs-of-cycles}
If 	$n\ge 7$, then $\mu (D(C_n)) = n$.
\end{theorem}

\begin{proof}
Let $V(D(C_n)) = V \cup V'$, where $V = V(C_n)$ and $V' = \{u':\ u \in V\}$. Let $S$ be a $\mu$-set of $D(C_n)$ such that it contains as many vertices of $V'$ as possible. This choice of $S$, together with Lemma~\ref{lem:false-twins}(i), implies that if $u \in V \cap S$, then $u'$ also belongs to $S$.

If $V'\subseteq S$, then no vertex from $V$ can be present in $S$, because if $v\in S$, then the two neighbors of $v'$ in $V'$  are not $S$-visible. Therefore, in this case $|S| = n$. By the same argument we also get that if $|S| > n$, then $|S \cap V| \geq 2$. We may hence assume in the rest that not all vertices from $V'$ are in $S$. We now distinguish two cases. 

Assume first that $S$ contains at least three vertices from $V$, say  $u, v, w \in V \cap S$. Then, by the maximality assumption, $u'$, $v'$ and $w'$ are also in $S$. Now, if $x'$ belongs to $S$, where $x \ne u, v, w$, then at least one of the shortest paths in $V'$ from $x'$ to $u'$, $v'$ or $w'$ must contain at least one vertex among $u'$, $v'$ and $w'$ as an internal vertex. We may assume without loss of generality that a shortest $x',u'$-path contains $v'$ as an internal vertex. (It could be that also the other $x',u'$-path in $D(C_n)[V']$ is shortest. Then it contains $w$ as an internal vertex, and the argument is parallel.) Since $v$ is also in $S$, the vertices $x'$ and $u'$ are not $S$-visible. Therefore, $|S\cap V'| = 3$ so that $|S|=6$, a contradiction with Theorem~\ref{thm:mu-of-double-graphs} which asserts that $|S| \ge n \ge 7$. 

Assume second that $S\cap V = \{u, v\}$. Using the maximality assumption again, $u', v'\in S$. There is nothing to prove if $|S\cap V'| \le n-2$, hence assume that $|S\cap V'| = n-1$. Let $w'\in V'$ be the vertex not in $S$. If $w$ is not adjacent to both $u$ and $v$, then we may assume without loss of generality that the two neigbors of $u'$ are in $S$, but then they are not $S$-visible. Similarly, if $w$ is in $C_n$ adjacent to both $u$ and $v$,  and $z$ is the other neighbor of $u$, then $z'$ and $v'$ are not $S$-visible. 

As none of the cases above is possible we can conclude that $\mu (D(C_n)) \le  n$ and so $\mu (D(C_n)) = n$. 
\end{proof}	

The proof of Theorem~\ref{thm:mu-of-double-graphs-of-cycles} asserts that for any $n \geq 4$, we have $\mu(D(C_n)) \leq \max \{6, \lfloor \frac{n}{2} \rfloor + 4, n\}$. Hence $\mu(D(C_4)) \leq 6$, $\mu (D(C_5)) \leq 6$ and $\mu(D(C_6)) \leq 7$. Also, $\{v_1', v_2', v_3', v_4', v_1, v_2\}$ is a mutual-visibility set of $D(C_4)$, $\{v_2', v_3', v_4', v_5', v_2, v_5\}$ is a mutual-visibility set of $D(C_5)$ and $\{v_2', v_3', v_4', v_5', v_6', v_2, v_6\}$ is a mutual-visibility set of $D(C_6)$. Therefore $\mu(D(C_4)) = \mu(D(C_5)) = 6$ and $\mu(D(C_6)) = 7$.\\	
 
While the lower bound of Theorem~\ref{thm:mu-of-double-graphs} is sharp, it can, on the other hand, be arbitrarily bad, that is, the difference $\mu (D(G)) - (n(G) + \mut(G))$ can be arbitrarily large. For example, consider the balloon graph $G_k$, $k\ge 2$, constructed from the disjoint union of $k$ copies of $C_5$ and a vertex which is adjacent to exactly one vertex of each of the $k$ copies of $C_5$. Then we have: 

\begin{proposition}
\label{prop:ballon}
If $k\ge 2$, then $\mu (D(G_k)) - (n(G_k) + \mut(G_k)) \geq k - 1$.
\end{proposition}

\proof
Clearly, $n(G_k) = 5k+1$. Using the characterization~\cite[Theorem 8]{tian-2024} of graphs $G$ with $\mut(G) = 0$ (or by verifying it directly), we can deduce that $\mut(G_k) = 0$. Further, we can use Fig.~\ref{fig:balloon-graph} to find out that $\mu (D(G_k)) \ge 6k$. 

\begin{figure}[ht!]
\centering
\includegraphics[width=0.9\linewidth]{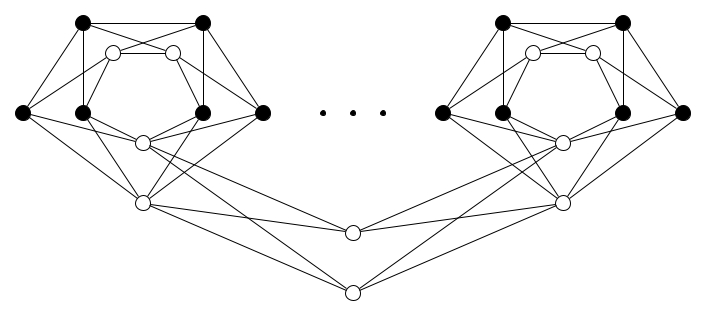}
	\caption{A mutual-visibility set in the double graph of a balloon graph.}
	\label{fig:balloon-graph}
\end{figure}

Hence we have 
$$\mu (D(G_k)) - (n(G_k) + \mut(G_k)) \geq 6k - ((5k+1) + 0) = k-1\,,$$
and we are done. 
\qed

\section{General position in double graphs}	
\label{sec:gp-double}

In this section we consider the general position number of double graphs. We will use the convention that if $S\subseteq V(G) \subset V(D(G))$, then $S' = \{u'\in V(G'):\ u\in S\}$.

We first state a simple but useful lemma which easily follows from the fact that if $u\in V(G) \subset V(D(G))$ is not an isolated vertex, then $d_{D(G)}(u,u') = 2$. 

\begin{lemma}\label{lem:twins-in-D(G)}
Let $G$ be a graph, $uv\in E(G)$, and let $S$ be a general position set of $D(G)$. If $|S\cap \{u,v,u',v'\}| \ge 2$, then $|S\cap \{u,v,u',v'\}| = 2$.
\end{lemma}

Note that Lemma~\ref{lem:twins-in-D(G)} in particular implies that if $S$ is a general position set of $D(G)$ such that $u, u' \in S$, then both $u$ and $u'$ are non-adjacent to all other vertices in $S$. 

\begin{theorem}
\label{thm:gp-of-double}
If $G$ is a graph, then $\gp(G) \leq \gp(D(G)) \leq 2\gp(G)$ and the bounds are sharp. Moreover, $\gp(D(G)) = 2\gp(G)$ if and only if the $gp$-sets of $D(G)$ are of the form $X \cup X'$, where $X$ is an independent $gp$-set of $G$.
\end{theorem}

\begin{proof}
If $S$ is a gp-set of $G$, then $S \subset V(D(G))$ is a general position set of $D(G)$. Hence $\gp(G) \leq \gp(D(G))$. Let now $S$ be a gp-set of $D(G)$. Since $G$ and $G'$ are isometric subgraphs of $D(G)$, we infer that $S \cap V(G)$ is a general position set of $G$ and $S \cap V(G')$ is a general position set of $G'$. Hence 
$$\gp (D(G)) = |S| = |S \cap V(G)| + |S \cap V(G')| \le \gp(G) + \gp(G') = 2 \gp(G)\,,$$
establishing the upper bound. 
   	
To see that the lower bound is sharp, note that $\gp(D(K_n)) = n$ holds for all $n\ge 2$ by Lemma~\ref{lem:twins-in-D(G)}. 

Assume now that $\gp(D(G)) = 2\gp(G)$ and consider an arbitrary gp-set $S$ of $D(G)$. As we already observed, $S \cap V(G)$ is a general position set of $G$ and $S \cap V(G')$ is a general position set of $G'$, therefore $S \cap V(G)$ is a gp-set of $G$ and $S \cap V(G')$ a gp-set of $G'$. If $(S \cap V(G))' \ne S \cap V(G')$, then we may assume without loss of generality that there exists a vertex $u\in S$ such that $u'\notin S$. But then an application of Lemma~\ref{lem:false-twins} yields a general position set in $G'$ larger that $\gp(G') = \gp(G)$, a contradiction. Hence $S = (S \cap V(G)) \cup (S \cap V(G))'$. Moreover, by Lemma~\ref{lem:twins-in-D(G)}, $S \cap V(G)$ must be an independent set and we are done. 
\end{proof}
	
There are many graphs admitting independent gp-sets which in turn explicitly demonstrate that the upper bound of Theorem~\ref{thm:mu-of-double-graphs} is sharp. This is in particular the case for paths $P_n$, $n\ge 3$, and for cycles $C_n$, $n\ge 6$. Hence by Theorem~\ref{thm:gp-of-double} we get $\gp(D(P_n)) = 4$, $n \geq 3$, and $\gp(D(C_n)) = 6$, $n \geq 6$. More on independent general position sets can be found in~\cite{thomas-2021}.

Another family of graphs for which the lower bound in Theorem~\ref{thm:gp-of-double} is sharp are the edge deleted complete graphs $K_n^-$, $n\ge 5$, that is, $K_n^-$ is the graph obtained from $K_n$ by deleting one of its edges. Note first that since $D(K_n^-)$ contain a clique of order $n-1$ we have $\gp(D(K_n^-)) \ge n-1$. Let $u$ and $v$ be the non-adjacent pair of vertices in $K_n^-$ and let $S$ be a general position set of $D(K_n^-)$. If $w, w'\in S$, where $w\ne u,v$, then $S = \{w,w'\}$. Assume hence that for each $w\ne u,v$, the set $S$ contains at most one vertex among $w$ and $w'$. If $|S| \ge n$, then we must have that $|S\cap \{u, u', v, v'\}| \ge 2$. However, as soon as this is fulfilled we can infer that in each possible case we have $S \subseteq \{u, v, u', v'\}$. We conclude that $\gp(D(K_n^-)) = n$ for $n\ge 5$.

\section{Mutual-visibility in Mycielskian graphs}
\label{sec:mycielski}

The general position number of Mycielskian graphs was investigated in~\cite{thomas-2024+}, in this section we complement this research by considering   the mutual-visibility number of Mycielskian graphs. We find the exact value of the mutual-visibility number of Mycielskian graph of paths, cycles and graphs with universal vertices. Bounds of mutual-visibility number of Mycielskian graph of graphs having diameter at most three in terms of outer mutual-visibility and mutual-visibility of the graph are also presented.

\begin{theorem}\label{thm:mu-of-mycielskian-graphs-of-paths}
If $n\ge 5$, then $\mu(M(P_n)) = n +  \lfloor \frac{n+1}{4}\rfloor$.
\end{theorem}

\begin{proof}
Let $R = \{v_1, v_3, v_5, \ldots, v_k\}$, where $k = n - 1$,  when $n$ is even, and $k = n$, when $n$ is odd. Let $R' = \{v_{4l+2}' : 0 \leq l \leq \lfloor \frac{n-3}{4} \rfloor\}$. If $|R|$ is even, then the last vertex in $R'$ is $v_{k-1}'$. If $|R|$ is odd, then the last vertex in $R'$ is $v_{k-3}'$, in which case we further add $v_{k-1}'$ to $R'$. Then, it is straightforward to verify that $S = R \cup (V(P_n') \setminus R')$ is a mutual-visibility set of $M(P_n)$. Since $|R| = \lceil \frac{n}{2} \rceil$ and $|R'| = \lceil \frac{1}{2}\lceil \frac{n}{2} \rceil \rceil$, we have $|S| = \lceil \frac{n}{2} \rceil + n-\lceil \frac{1}{2}\lceil \frac{n}{2} \rceil \rceil = n +  \lfloor \frac{n+1}{4}\rfloor$. Thus, $\mu(M(P_n)) \geq  n +  \lfloor \frac{n+1}{4}\rfloor$. 

In the rest of the proof we need to show that $\mu(M(P_n)) \le n +  \lfloor \frac{n+1}{4}\rfloor$. We first show the conclusion holds if some $\mu$-set of $M(P_n)$ contains $v^*$. Let hence $N \cup N'$ be a $\mu$-set of $M(P_n)$, where $N \subseteq V(P_n)$, $N' \subseteq V(P_n') \cup \{v^*\}$, and $v^* \in N'$. Since $v^*$ is in the unique shortest path connecting $u'$ and $v'$, where $d_{P_n}(u,v) \neq 2$, at most two vertices from $V(P_n')$ are in $N'$. Also, since $v^*$ is in the unique shortest path connecting $u$ and $v$, where $d_{P_n}(u,v) \geq 5$, at most four vertices from $V(P_n)$ are in $N$. Hence we have $|N \cup N'| \leq 7$ which proves the assertion for $n \geq 7$. For $n = 5, 6$, using similar arguments we can prove that if $v^* \in N'$ then $|N \cup N'| \leq 5$ and $|N \cup N'| \leq 6$, respectively. 

In the following we may thus reduce our attention to $\mu$-sets which do not contain the vertex $v^*$. 
	
\medskip\noindent
{\bf Claim}: There exists a $\mu$-set $S \cup S'$ of $M(P_n)$, where $S \subseteq V(P_n)$, $S' \subseteq V(P_n')$, such that $S$ is an independent set.\\
Let $N \cup N'$ be a $\mu$-set of $M(P_n)$, where $N \subseteq V(P_n)$ and $N' \subseteq V(P_n')$. If $N$ is independent, there is nothing to prove. Otherwise, proceed as follows to replace the vertices in the $\mu$-set so as to make a new mutual-visibility of the same cardinality and which is independent restricted to $P_n$. The construction is distinguished according to the following situations.
	
Assume first that three consecutive vertices of $P_n$ lie in $N$. If  $v_{k-1}, v_{k}, v_{k+1} \in N$, where $2 < k < n-2$, then none of the vertices $v_{k-2}$, $v_{k-2}'$, $v_{k-1}'$, $v_{k}'$, $v_{k+1}'$, $v_{k+2}$, $v_{k+2}'$ lies in $N \cup N'$, see Fig.~\ref{fig:fig6}(a). Then we infer that $(N \cup N' \cup \{v_{k}'\})\setminus\{v_{k}\}$ is a mutual-visibility set of $M(P_n)$ and hence a $\mu$-set of $M(P_n)$. If $v_{1}, v_{2}, v_{3} \in N$, then none of $v_{1}'$, $v_{2}'$, $v_{3}'$, $v_{4}$, $v_{4}'$ lies in $N \cup N'$. In this case we see that $(N \cup N' \cup \{v_{1}'\})\setminus\{v_{2}\}$ is a mutual-visibility set of $M(P_n)$. Similarly, if $v_{n-2}, v_{n-1}, v_{n} \in N$, then $(N \cup N' \cup \{v_{n}'\})\setminus\{v_{n-1}\}$ is a mutual-visibility set of $M(P_n)$. We have this seen that $N \cup N'$ can be modified in such a way that no three consequent vertices from $P_n$ are in $N$.
	
	\begin{figure}[ht!]
		\centering
		\includegraphics[width=0.9\linewidth]{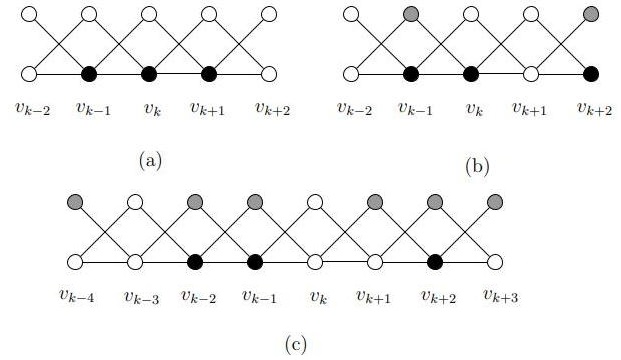}
		\caption{Situations from the proof of Theorem~\ref{thm:mu-of-mycielskian-graphs-of-paths}. The black vertices denote the vertices in the mutual-visibility set, the grey vertices are those whose status is not known and the white vertices are those which cannot be present in the mutual-visibility set.}
		\label{fig:fig6}
	\end{figure}
	
Assume next that $v_{k-1}, v_k, v_{k+2} \in N$, where $2 < k < n-2$. Then the vertices $v_{k-2}$, $v_{k-2}'$, $v_{k}'$, $v_{k+1}$, $v_{k+1}'$ do no belong to $N \cup N'$, cf.\ Fig.~\ref{fig:fig6}(b), where $v_{k-2}$ and $v_{k+1}$ are not present by the above modification. Then $(N \cup N' \cup \{v_{k}'\})\setminus\{v_{k}\}$ is a mutual-visibility set of $M(P_n)$. If $v_{1}, v_{2}, v_{4} \in N$, then $v_{1}', v_{2}', v_{3}, v_{3}' \notin N \cup N'$. Then $(N \cup N' \cup \{v_{1}'\})\setminus\{v_{2}\}$ is a mutual-visibility set of $M(P_n)$. Similarly, if $v_{n-3}, v_{n-1}, v_{n} \in N$, then $(N \cup N' \cup \{v_{n}'\})\setminus\{v_{n-1}\}$ is a mutual-visibility set of $M(P_n)$. We can thus further modify $N \cup N'$ in such a way that no three vertices from $P_n$ of the form $v_{k-1}, v_k, v_{k+2}$ are in $N$.
	
In the third case assume that $v_{k-2}, v_{k-1}, v_{k+2} \in N$, where $ k \neq 3, n-2$. Then $v_{k-3}', v_{k}' \notin N \cup N'$. Moreover, we also infer  that by the above modifications, $v_{k-3}, v_{k}, v_{k+1} \notin N$, see Fig.~\ref{fig:fig6}(c). Then $(N \cup N'\cup \{v_{k}, v_{k}'\})\setminus\{v_{k-1}, v_{k + 1}', v_{k+3}\}$ is a mutual-visibility set of $M(P_n)$. (Note that only one among $v_{k+1}'$ and $v_{k+3}$ will be present initially in $N \cup N'$ which implies that the cardinality of the mutual-visibility set remains the same.)  If $v_{1}, v_{2}, v_{5} \in N$, then $v_{1}', v_{2}', v_{3}, v_{3}', v_{4} \notin N \cup N'$. Then $(N \cup N' \cup \{v_3, v_{3}'\})\setminus\{v_{4'}, v_6\}$ is a mutual-visibility set of $M(P_n)$. (Note that only one among $v_{4}'$ and $v_{6}$ will be present initially in $N \cup N'$.) Similarly, if $v_{n-4}, v_{n-1}, v_{n} \in N$, then $(N \cup N' \cup \{v_{n-2}, v_{n-2}'\})\setminus\{v_{n-1}, v_{n-3}', v_{n-5}\}$ is a mutual-visibility set of $M(P_n)$. (Note that only one among $v_{n-3}'$ and $v_{n-5}$ will be present initially in $N \cup N'$.) 

In the last case to be considered assume that $v_{k-2}, v_{k-1}, v_{k+l} \in N$ for some $l \geq 3$. Then $v_{k}, v_{k}', v_{k+1}, v_{k+2} \notin N \cup N'$. In addition, $v_{k-5}, v_{k-4}, v_{k-3}, v_{k-3}'\notin N\cup N'$, if those vertices are present in the graph. Then $(N \cup \{v_k\}) \setminus \{v_{k-1}\}$ is a mutual-visibility set of $M(P_n)$. Thus modify $N \cup N'$ in such a way that no three vertices from $P_n$ of this form are in $N$.
	
If $N$ has only two vertices and they are adjacent, then at least two vertices from $P_n'$ are not in $N'$, in which case, $|N \cup N'| \leq n$. This is not possible since $\mu(M(P_n)) \geq  n +  \lfloor \frac{n+1}{4}\rfloor > n$. This proves the claim. 
	
\medskip	

We have thus proved that there exists a $\mu$-set $S \cup S'$ of $M(P_n)$, where $S \subseteq V(P_n)$, $S' \subseteq V(P_n')$, such that $S$ is an independent set.	Now, if there are two vertices $u$ and $v$ in $S$ such that $d_G(u,v) \geq 5$, then each of $u$ and $v$ must have a neighbor in $V(P_n')$ which is not in $S'$ in order that $u$ are $v$ are visible. Also, if there are two vertices $u$ and $v'$ in $S$ such that $d_G(u,v) = 4$, then $u$ must have a neighbor in $V(P_n')$ which is not in $S'$ in order that $u$ are $v'$ are visible. Such a vertex from $V(P_n')\setminus S'$ can be used or shared by at most two vertices from $S$. Hence, the cardinality of $S \cup S'$ will be maximum when the vertices $v_1, v_3, v_5, \ldots$ are included into $S$ and the vertices $v_2',v_6', v_{10}',\dots$ from $V(P_n')$ are excluded from $S'$. Hence, $\mu(M(P_n)) \leq \lceil \frac{n}{2} \rceil + n-\lceil \frac{1}{2}\lceil \frac{n}{2} \rceil \rceil =  n +  \lfloor \frac{n+1}{4}\rfloor$ and we are done. 
\end{proof}

\begin{theorem} \label{thm:mu-of-mycielskian-graphs-of-cycles}
If $n\ge 8$, then $\mu(M(C_n)) =  n + \lfloor \frac{n}{4} \rfloor$.
\end{theorem}

\begin{proof}
	Let $R \subseteq V(C_n)$ be an independent set of maximum cardinality of $C_n$ and let $R'$ be a smallest set of vertices from $V(C_n')$ which dominate all the vertices from $R$. Then it is straightforward to verify that $R \cup (V(C_n') \setminus R')$ is a mutual-visibility set of $M(C_n)$. Since $|R| = \lfloor \frac{n}{2} \rfloor$ and $|R'| = \lceil \frac{1}{2}\lfloor \frac{n}{2} \rfloor \rceil$ we have $|R \cup (V(C_n') \setminus R')| = \lfloor \frac{n}{2} \rfloor + n-\lceil \frac{1}{2}\lfloor \frac{n}{2} \rfloor \rceil  = n + \lfloor \frac{n}{4} \rfloor$ which in turn implies that $\mu(M(C_n)) \geq  n + \lfloor \frac{n}{4} \rfloor$. 
	
To prove that $\mu(M(C_n)) \leq n + \lfloor \frac{n}{4} \rfloor$, we first show that the inequality holds if some $\mu$-set of $M(C_n)$ contains $v^*$. So let $N \cup N'$ be a $\mu$-set of $M(C_n)$, where $N \subseteq V(C_n)$, $N' \subseteq V(C_n') \cup \{v^*\}$, and $v^*\in N'$. Since $v^* \in N'$, at most two vertices from $V(C_n')$, say $v_k'$ and $v_{k+2}'$, are in $N'$. If $n \in \{8, 9\}$, it can easily be verified that $|N \cup N'| \leq 8$. For $n \geq 10$, at most five vertices from $V(C_n)$ are in $N$, since no two vertices $u, v \in V(C_n)$ such that $d_{C_n}(u,v) \geq 5$, are visible. So in any case the claimed inequality holds. In the rest we may thus assume that the $\mu$-sets of $M(C_n)$ do not contain $v^*$. We also assume that indices are computed modulo $n$. 

\medskip\noindent
{\bf Claim}: There exists a $\mu$-set $S \cup S'$ of $M(C_n)$, where $S \subseteq V(C_n)$ and $S' \subseteq V(C_n')$, such that if $v_k, v_{k+1} \in S$, then $v_{k-2}, v_{k-1}, v_{k+2}, v_{k+3} \notin S$.

Let $N \cup N'$ be a $\mu$-set of $M(C_n)$, where $N \subseteq V(C_n)$ and $N' \subseteq V(C_n')$. Then we are going to modify $N \cup N'$ such that the modified $\mu$-set of $M(C_n)$ will satisfy the condition of the claim. To this end, we distinguish a few cases. 

Assume first that $v_{k-1}, v_{k}, v_{k+1} \in N$. In this case, $v_{k-2}$, $v_{k-2}'$, $v_{k-1}'$, $v_{k}'$, $v_{k+1}'$, $v_{k+2}$, and $v_{k+2}'$ do not lie in $N \cup N'$, just as shown in Fig.\ \ref{fig:fig6}(a). Now we consider the set $(N \cup N' \cup \{v_{k}'\})\setminus\{v_{k}\}$ and show that it is a mutual-visibility set of $M(C_n)$. If $v_{k+l} \in N$, for $l = 3$ or for any $l \geq 5$, then since $v_{k}$ and $v_{k+l}$ are visible, we get $v_{k}'$ and $v_{k+l}$ are visible. Similarly, if $v_{k-l} \in N$, for $l = 3$ or for any $l \geq 5$, we get that $v_{k}'$ and $v_{k-l}$ are visible. Now, if $n \geq 10$, then $v_k'$ and $v_{k+4}$ are visible since $v_{k-1}$ and $v_{k+4}$ are visible. Similarly, $v_k'$ and $v_{k-4}$ are visible because $v_{k-1}$ and $v_{k-4}$ are visible. It can be verified directly that for $n=8$, the above possibility along with $v_{k-4} \in N$ or $v_{k+4} \in N$ implies $|N\cup N'| \leq 7$, which is a contradiction. Similarly, for $n=9$, the above possibility along with $v_{k-4} \in N$ or $v_{k+4} \in N$ implies $|N\cup N'| \leq 10$, which is a contradiction. 

	\begin{figure}[ht!]
		\centering
		\includegraphics[width=0.8\linewidth]{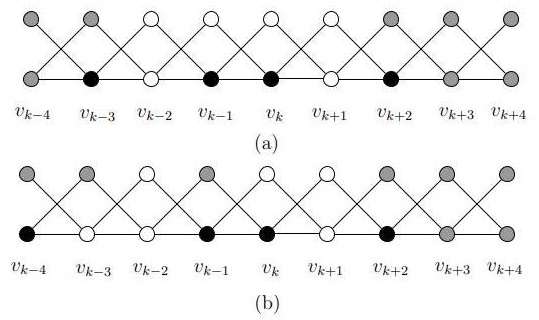}
		\caption{Situations from the proof of Theorem~\ref{thm:mu-of-mycielskian-graphs-of-cycles}. The black vertices again denote the vertices in the mutual-visibility set, the grey vertices are those whose status is not known and the white vertices are those which cannot be present in the mutual-visibility set.}
		\label{fig:fig7}
	\end{figure}
	
Assume second that $v_{k-3}, v_{k-1}, v_k, v_{k+2} \in N$. In this case, $v_{k-2}$, $v_{k-2}'$, $v_{k-1}'$, $v_{k}'$, $v_{k+1}$, and $v_{k+1}'$ do not belong to $N \cup N'$, see Fig.~\ref{fig:fig7}(a). We now consider $(N \cup N' \cup \{v_{k}'\})\setminus\{v_{k}\}$ and assert that it is a mutual-visibility set of $M(C_n)$. The vertices $v_{k}'$ and $v_{k+2}$ are visible since $v_{k+1} \notin N$. If $v_{k+l} \in N$, for $l = 3$ or for any $l \geq 5$ then, since $v_{k}$ and $v_{k+l}$ are visible, we get that $v_{k}'$ and $v_{k+l}$ are visible. Similarly, if $v_{k-l} \in N$, for $l = 3$ or for any $l \geq 5$, we get that $v_{k}'$ and $v_{k-l}$ are visible. If $v_{k-4} \in N$, then either $v_{k-5}'$ or $v_{k-3}'$ not in $N'$. Hence, $v_k'$ and $v_{k-4}$ are visible. If $n \geq 10$, then $v_k'$ and $v_{k+4}$ are visible since $v_{k-1}$ and $v_{k+4}$ are visible. It can be directly verified that for $n=9$, the above possibility along with $v_{k+4} \in N$ implies $|N\cup N'| \leq 10$. 
	
Assume next that $v_{k-4}, v_{k-1}, v_k, v_{k+2} \in N$. In this case, $v_{k-3}$, $v_{k-2}$, $v_{k-2}'$, $v_{k}'$, $v_{k+1}$, $v_{k+1}'$ do not lie in $N \cup N'$, cf.\ Fig.~\ref{fig:fig7}(b)). We now consider two subcases. 

In the first subcase assume that $v_{k-3}' \notin N'$ or $v_{k-5}' \notin N'$. Then we assert that $(N \cup N' \cup \{v_{k}'\})\setminus\{v_{k}\}$ is a mutual-visibility set of $M(C_n)$. The vertices $v_{k}'$ and $v_{k+2}$ are visible since $v_{k+1} \notin N$. If $v_{k+l} \in N$, for $l = 3$ or for any $l \geq 5$, then, since $v_{k}$ and $v_{k+l}$ are visible, we get $v_{k}'$ and $v_{k+l}$ are visible. Similarly, if $v_{k-l} \in N$, for $l = 3$ or for any $l \geq 5$, we get $v_{k}'$ and $v_{k-l}$ are visible. Also, $v_k'$ and $v_{k-4}$ are visible since $v_{k-3}' \notin N'$ or $v_{k-5}' \notin N'$. Now, for $n \geq 10$, the vertices $v_k'$ and $v_{k+4}$ are visible since $v_{k-1}$ and $v_{k+4}$ are visible. It can be easily verified that for $n=9$, the above possibility along with $v_{k+4} \in N$ implies $|N\cup N'| \leq 10$
	
In the second subcase assume that $v_{k-5}', v_{k-3}'\in N'$. Then $v_{k-1}' \notin N'$, since $v_k$ and $v_{k-4}$ are visible. We now asserts that $(N \cup N' \cup \{v_{k-1}'\})\setminus\{v_{k-1}\}$ is a mutual-visibility set of $M(C_n)$. If $v_{k+l} \in N$, for $l = 2$ or for any $l \geq 4$, then, since $v_{k-1}$ and $v_{k+l}$ are visible, we get that $v_{k-1}'$ and $v_{k+l}$ are visible. Similarly, if $v_{k-l} \in N$, for $l = 4$ or for any $l \geq 6$, we find that $v_{k-1}'$ and $v_{k-l}$ are visible. Also, $v_k'$ and $v_{k+3}$ are visible since $v_{k+2}' \notin N'$ or $v_{k+4}' \notin N'$. Similarly, $v_k'$ and $v_{k-5}$ are visible since $v_{k-4}' \notin N'$ or $v_{k-6}' \notin N'$. The claim is proved. 

\medskip
We have thus proved that there exists a $\mu$-set $S \cup S'$ of $M(C_n)$, where $S \subseteq V(C_n)$ and $S' \subseteq V(C_n')$, such that if $v_k, v_{k+1} \in S$, then $v_{k-2}, v_{k-1}, v_{k+2}, v_{k+3} \notin S$. We are now going to show that for this set we have $|S \cup S'| \leq \lfloor \frac{n}{2} \rfloor + n-\lceil \frac{1}{2}\lfloor \frac{n}{2} \rfloor \rceil$. If $v_k, v_{k+1} \in S$, then $v_{k-2}$, $v_{k-1}$, $v_{k+2}$, $v_{k+3}$ do not belong to $S$. Also, $v_{k-1}', v_{k+2}' \notin S'$. Now, if $S' = V(C_n')$, then $S = \emptyset$, for $n \geq 8$. If $S' = V(C_n') \setminus \{v_2'\}$, then $S = \{v_1, v_3\}$. If $S' = V(C_n') \setminus \{v_2', v_6'\}$, then $S = \{v_1, v_3, v_5, v_7\}$. This process can be continued until the $\lceil \frac{1}{2}\lceil \frac{n}{2} \rceil \rceil$ vertices $v_2',v_6', v_{10}',\ldots$ are excluded from $V(C_n')$ so that the $\lfloor \frac{n}{2} \rfloor$ vertices $v_1, v_3, v_5,\ldots$ can be included into $S$. Hence $\mu(M(C_n)) \le \lfloor \frac{n}{2} \rfloor + n-\lceil \frac{1}{2}\lfloor \frac{n}{2} \rfloor \rceil =  n + \lfloor \frac{n}{4} \rfloor$ and we are done.
\end{proof}

\begin{proposition}\label{28}
	If $G$ is a graph with $n(G)\ge 2$ and with a universal vertex, then $\mu(M(G)) = 2n(G)-1$.
\end{proposition}

\begin{proof}
	Let $v$ be a universal vertex of $G$. Then it is straightforward to verify that $(V(G)\setminus\{v\}) \cup V(G')$ is a mutual-visibility set of $M(G)$ which implies that $\mu(M(G)) \ge 2n(G)-1$. 
	
Let $S$ be an arbitrary mutual-visibility set of $M(G)$. If $V(G) \subseteq S$, then $S \cap V(G') = \emptyset$, hence in this case $|S| \le n(G) + 1 \le 2n(G)-1$. The same conclusion holds when $S\cap V(G) = \emptyset$. Assume in the rest that $1\le |S\cap V(G)| \le n(G)-1$. Then if $v^*\notin S$, we immediately get $|S| \le 2n(G)-1$ and if $v^*\in S$, then $|S\cap V(G')| \le n(G) - 1$, and we obtain the same conclusion. In any case $\mu(M(G)) \ge 2n(G)-1$. 
\end{proof}

\begin{theorem} \label{thm:diam3}
If $G$ is not a complete graph and $\diam(G) \leq 3$, then 
$$n(G) + \muo(G) \leq \mu (M(G)) \leq n + \mu(G) + 1\,.$$ 
Moreover, if $\mu (M(G)) = n(G) + \mu(G) + 1$, then every $\mu$-set of $M(G)$ contains $v^*$.
\end{theorem}

\begin{proof}
Let $M$ be a $\muo$-set of $G$. Then, having in mind that $\diam(G)\le 3$,  it is straightforward to verify that $M \cup V(G')$ is a mutual-visibility set of $M(G)$ and therefore $\mu (M(G)) \geq n + \muo(G)$. 
	
Let $S$ be a $\mu$-set of $M(G)$. Assume first that $S \cap V(G)$ is a mutual-visibility set of $G$. Then $|S \cap V(G)| \leq \mu(G)$ and hence, $|S| \leq n + \mu(G) + 1$. Assume second that $S \cap V(G)$ is not a mutual-visibility set of $G$. Then there exists $u, v \in S \cap V(G)$ such that $u$ and $v$ are not mutually visible in $G$ but are mutually visible in $M(G)$. Denoting by $I_G[u,v]$  the set of all vertices that lie on shortest $u,v$-paths in $G$, we then have $x \in I_G[u,v]$ such that $x \in S$ but $x' \notin S$. If $(S \setminus \{x\}) \cap V(G)$ is a mutual-visibility set of $G$ then $|S \cap V(G)| \leq \mu(G) + 1$. Hence, $|S| \leq n + \mu(G) + 1$. If $(S \setminus \{x\}) \cap V(G)$ is not a mutual-visibility set of $G$ then proceed as above, that is, at each step we detect a vertex $y'\in V(G')\setminus S$ corresponding to a vertex $y\in S$. Therefore, $\mu (M(G)) \leq n + \mu(G) + 1$.
	 
Assume now that $\mu (M(G)) = n + \mu(G) + 1$ and suppose by way of contradiction that $S$ be a $\mu$-set of $M(G)$ such that $v^*\notin S$. Then $|S \cap V(G)| = \mu(G) + k$, for some $k \geq 1$. It follows that $|S \cap V(G')| = n(G) - k + 1$. Using a parallel argument as above, we are now going to show that $|S \cap V(G')| \leq n(G) - k$, which becomes a contradiction. Since $|S \cap V(G)|$ is not a mutual-visibility set of $G$, there exists $u, v \in S \cap V(G)$ such that $u$ and $v$ are not mutually visible in $G$ but are mutually visible in $M(G)$. Then there exists $x \in I_G[u,v]$ such that $x \in S$ but $x' \notin S$. If $k \neq 1$ then $(S \cap V(G)) \setminus \{x\}$ is still not a mutual-visibility set of $G$ and hence the above process can be repeated. Thus there exists distinct vertices $x_1', \ldots , x_k'$ that are not in $S\cap V(G')$ and hence $|S \cap V(G')| \leq n(G) - k$.
\end{proof}

We now give some examples how Theorem~\ref{thm:diam3} can be applied. First, from the theorem we read that $\mu(M(P_4)) \in \{6,7\}$ and that, moreover, if $\mu(M(P_4)) = 7$, then every$\mu$-set of $M(P_4)$ contains $v^*$. But if a mutual-visibility set $S$ of $M(P_4)$ contains contains $v^*$, then we infer that $|S| \leq 5$. We can conclude that $\mu(M(P_4)) = 6$.

Let $r_1 \geq r_2 \geq 3$ and set $n = r_1 + r_2$. Using Theorem~\ref{thm:diam3} we get that $\mu(M(K_{r_1, r_2})) \in \{2n - 2, 2n - 1\}$ and that if $\mu(M(K_{r_1, r_2 })) = 2n - 1$, then every $\mu$-set of $M(K_{r_1, r_2})$ contains $v^*$. Let $S$ be an arbitrary mutual-visibility set of $M(K_{r_1, r_2})$ with $v^*\in S$, and let $uv\in E(K_{r_1, r_2})$. Then at most one among $u'$ and $v'$ can be in $S$, hence $|S| \leq n + r_1 - 1$. We conclude that $\mu(M(K_{r_1, r_2})) = 2n - 2$.

In Theorem~\ref{thm:mu-of-mycielskian-graphs-of-cycles} we have determined $\mu(M(C_n))$ for $n\ge 8$. We  now do the same for shorter cycles. By Theorem~\ref{thm:diam3}, $\mu(M(C_n)) \geq n + 2$, for $4 \leq n \leq 7$. We claim that here equality always holds. Let $S$ be a $\mu$-set of $M(C_n)$. If $v^* \in S$ then at most two vertices from $V(C_n')$, say $v_k'$ and $v_{k+2}'$, are in $S$. If $v_k', v_{k+2}' \in S$ then at least one vertex, say $v_{k+1}$, is not in $S$. Thus in this case, $|S| \leq n+ 2$. Now, suppose $v^* \notin S$. Then as in the proof of Theorem~\ref{thm:mu-of-mycielskian-graphs-of-cycles}, there exists a $\mu$-set $S$ of $M(C_n)$ such that if $v_k, v_{k+1} \in S$, then $v_{k-2}, v_{k-1}, v_{k+2}, v_{k+3} \notin S$. If $V(C_n') \not \subseteq S$, then $|S| \leq n + \lfloor \frac{n}{4} \rfloor = n$. If $V(C_n') \subseteq S$, then $S \cap V(C_n)$ is an outer mutual-visibility set of $C_n$ and hence $|S| \leq n + 2$. Therefore, $\mu(M(C_n)) = n + 2$, for $4 \leq n \leq 7$.
 
\section*{Acknowledgments}

Dhanya Roy thank Cochin University of Science and Technology for providing financial support under University JRF Scheme. Sandi Klav\v zar was supported by the Slovenian Research Agency ARIS (research core funding P1-0297 and projects J1-2452, N1-0285). 

\end{document}